\documentclass[12pt]{amsart}
\pdfoutput=1
\usepackage[margin=1.25in]{geometry}
\usepackage{amsmath}
\usepackage{graphicx}
\usepackage{latexsym}
\usepackage{color}
\usepackage{amscd}
\usepackage[all]{xy}
\usepackage{enumerate}
\usepackage{hyperref}
\usepackage{soul}
\usepackage{comment}
\usepackage{epsfig}
\usepackage{epstopdf}
\usepackage{tikz}
\usepackage{marginnote}
\usepackage{lipsum}
\usepackage{mathtools}
\usepackage{caption} 
\usepackage{subcaption}

\usepackage{multirow}

\newtheorem{thm}{Theorem}

\newtheorem{prop}[thm]{Proposition}

\newtheorem{theorem}[thm]{Theorem}
\newtheorem{lemma}[thm]{Lemma}

\newtheorem{question}[thm]{Question}

\theoremstyle{definition}

\newtheorem*{definition*}{Definition}

\newtheorem{remark}[thm]{Remark}

\newcommand{\CPb}{\overline{\mathbb{CP}}{}^{2}}
\newcommand{\CP}{{\mathbb{CP}}{}^{2}}

\newcommand{\R}{\mathbb{R}}
\newcommand{\Q}{\mathbb{Q}}

\newcommand{\Z}{\mathbb{Z}}

\newcommand{\M}{\operatorname{Mod}}
 \def\R{{\mathbb{R}}}
 \def\Z{{\mathbb{Z}}}
 \def\Q{{\mathbb{Q}}}

\begin{document}

\title[Geography of symplectic Lefschetz fibrations  and rational blowdowns]
{Geography of symplectic Lefschetz fibrations \\ and rational blowdowns}

\author[R. \.{I}. Baykur]{R. \.{I}nan\c{c} Baykur}
\address{Department of Mathematics and Statistics, University of Massachusetts, Amherst, MA 01003-9305, USA}
\email{baykur@math.umass.edu}

\author[M. Korkmaz]{Mustafa Korkmaz}
\address{Department of Mathematics, Middle East Technical University, 06800 Ankara, Turkey}
\email{korkmaz@metu.edu.tr}

\author[J. Simone]{Jonathan Simone}
\address{School of Mathematics, Georgia Institute of Technology, Atlanta, GA 30332, USA}
\email{jsimone7@gatech.edu}

\begin{abstract}
We produce  simply connected,  minimal,  symplectic Lefschetz fibrations realizing all the lattice points in the symplectic geography plane below the Noether line.  This provides a \emph{symplectic} extension of the classical works populating the complex geography plane with holomorphic Lefschetz fibrations.  Our examples are obtained by rationally blowing down Lefschetz fibrations with clustered nodal fibers,  the total spaces of which are potentially new homotopy elliptic surfaces.  Similarly,   clustering nodal fibers on higher genera Lefschetz fibrations on standard rational surfaces, we get rational blowdown configurations that yield new constructions of small symplectic exotic $4$--manifolds.  We present an  example of a construction of a minimal symplectic exotic  $\CP \# \, 5 \CPb$ through this procedure applied to a genus--$3$ fibration. 
\end{abstract}

\maketitle

\setcounter{secnumdepth}{2}
\setcounter{section}{0}

% ============================================================

\section{Introduction} 

The \emph{symplectic geography problem}, inspired by the study of compact complex algebraic surfaces  by Persson et.\,al.   \cite{persson81, persson87, perssonetalspin},  asks which pairs of integers $(a,b)$ can be realized as the holomorphic Euler characteristic $\chi_h=a$ and the first Chern number $c_1^2=b$ of a  closed minimal symplectic \mbox{$4$--manifold} \cite{gompf, mccarthywolfson}.  It is well-known that these invariants depend only on the underlying homotopy type of the  \mbox{$4$--manifold}  $X$,  satisfying the identities $\chi_h = \frac{1}{4}(\chi+\sigma)$ and  $c_1^2=2\chi+3\, \sigma$, where $\chi$ and $\sigma$ are the Euler characteristic and the signature of $X$.  Both coordinates are positive for  minimal simply connected $4$--manifolds of general type,  which are usually the  focus of the geography problem. 

In the case of compact complex algebraic surfaces,  the geography plane was populated almost exclusively by surfaces that are the total spaces of singular fibrations over complex curves \cite{persson81, persson87,perssonetalspin, zchen87, zchen91, sommese, roulleauurzua}, where most lattice points in the region $8 a \geq  b \geq 2a -6$ are realized by minimal holomorphic Lefschetz fibrations; see e.g.  \cite{perssonetalspin}.  The associated invariants of a compact complex surface of general type satisfy both the \textit{Bogomolov-Miayoka-Yau inequality} $9 \chi_h \geq c_1^2$ and  the \textit{Noether inequality} $ c_1^2 \geq 2 \chi_h -6$.  

Perhaps the most striking difference between the complex and symplectic geography is that the Noether inequality fails for symplectic $4$--manifolds \cite{gompf, fintushelparkstern,  ABBKP, sakalli}.  
% (whereas it is still not known if the Bogomolov-Miyaoka-Yau inequality extends over symplectic $4$--manifolds).
However, in this case,  there are only sporadic examples realizing lattice points in the region $ 2a -6 \geq b > 0$ as  Lefschetz fibrations, which were suggested by Fintushel and Stern in \cite{fintushelsternLFs}; see Remark~\ref{FSexamples}. The families of examples we produce in our first theorem will in particular contain minimal simply-connected Lefschetz fibrations populating this entire region. The interested reader can  refer to  Remark~\ref{LFgeography} for a further discussion on the geography of symplectic Lefschetz fibrations.

Before we formulate the theorem, let us recall the rational blowdown operation.  The \textit{rational blowdown} is performed by cutting out a regular neighborhood of a configuration of spheres, called a plumbing, embedded in an ambient $4$--manifold and replacing it with a rational homology $4$--ball.  Since the rational blowdown was  introduced by Fintushel-Stern in \cite{fintushelstern} and generalized by Park in \cite{parkbdown}, it has been used to construct many  exotic $4$--manifolds; e.g. \cite{stipsiczszabo,  park, parkexotic3cp, parkstipsiczszabo, yasui, parkparkshin4}, most of which are symplectic due to Symington \cite{symington}, \cite{symington2}. The appeal of this construction, along with other similar symplectic cut-and-paste operations (e.g. \cite{karakurtstarkston, simone}), is that the fundamental group and Seiberg-Witten invariant calculations are relatively routine.  The first step in this process is locating a suitable configuration of spheres in some $4$--manifold.  Partially because there is a well-known classification of singular fibers on elliptic fibrations, elliptic fibrations have been the most popular starting place to search for plumbings that can be symplectically rationally blown down (e.g.  most of the above referenced articles start with elliptic fibrations). More recently, Akhmedov and Sakall{\i} utilized the classification of singular fibers on certain holomorphic genus--$2$ Lefschetz fibrations for rational blowdowns
\cite{akhmedovsakalli}.

\begin{figure}
	\centering
	\includegraphics[scale=.5]{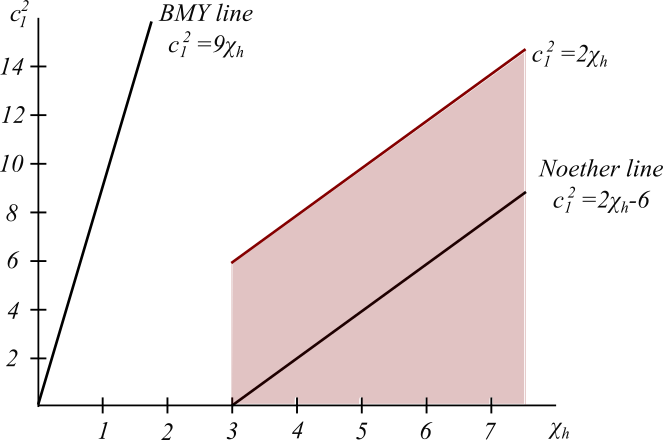}
	\caption{The lattice points in the shaded region $\mathcal{R}$ are populated by our minimal simply-connected symplectic Lefschetz fibrations.}
	\label{fig:gplane}
\end{figure}
%	For every point $(a,b)\in\Z^2$ in the shaded region $\mathcal{R}$, there is a minimal symplectic 4-manifold satisfying $(c_1^2,\chi_h)=(a,b)$ obtained by performing rational blowdowns along $-4$-spheres embedded in the twisted fiber sums $Z_g$.

In this article, our starting point will be genus $g >1$ Lefschetz fibrations on rational surfaces and homotopy elliptic surfaces, corresponding to positive factorizations in the mapping class group $\M(\Sigma_g)$. The particular positive factorizations we use allow us to cluster nodal singularities and obtain some extremal configurations for our rational blowdown procedures.  Thus, our first set of results will be on clustering nodal singularities in genus--$g$ Lefschetz fibrations on $\CP\#(4g+5)\CPb$; see Lemmas~\ref{lem:genus=g} and~\ref{lem:genus=3}.  Coupling these ideas with twisted fiber sums,  for each $g \geq 2$, we construct minimal symplectic genus--$g$ Lefschetz fibrations $(Z_{g+1},  f_{g+1})$, where $Z_{g+1}$ is a homotopy elliptic surface $E(g+1)$.  These  will contain  $2g+2$   disjoint embedded symplectic $(-4$)--spheres on the fibers we can then rationally blow down to get our first result.  Consider the region
$$\mathcal{R}=\{(a,b)\in\Z^2\text{ }|\text{ } a\ge 3 \text{ and } 0 < b\le 2a\}$$
shown in Figure~\ref{fig:gplane}. We have:

\begin{theorem}
	 For each point $(a,b)\in\mathcal{R}$, there exists a minimal simply-connected non-spin symplectic genus $g=a-1$ Lefschetz fibration $(Z_{a,b}, f_{a,b})$ satisfying $\chi_h(Z_{a,b})=a$ and $c_1^2(Z_{a,b})=b$ obtained from $(Z_{a}, f_a)$ by rationally blowing down  $b$ many $(-4)$--spheres contained in the fibers.
	\label{thm:gplane}
\end{theorem}

\noindent Similar to the regular blowdown, each  time we rationally blow down a $(-4)$-sphere,  $c_1^2$ increases by one whereas $\chi_h$ does not change.  So, we see that the portion of the geography plane highlighted in Figure~\ref{fig:gplane} is populated by $Z_{a,b}$ as we vary $(a,b) \in \mathcal{R}$. 
%setting $(Z_a, f_a):=(Z_{a,0}, f_{a,0})$, 

Curiously, the homotopy elliptic surfaces $Z_{g+1}$ we build are not diffeomorphic to well-known homotopy elliptic surfaces obtained from $E(g+1)$ by logarithmic transforms or  knot surgery on an elliptic fiber for any $g > 3$; we show this in Proposition~\ref{newishhomotopye(n)}.  The symplectic $4$--manifolds $Z_{a,b}$
with $0 < b < a-3$ are moreover interesting in connection to a conjecture of Fintushel and Stern on the number of basic classes recently proven by Feehan and Leness \cite{feehanleness}; see Remark~\ref{superconformal}.

Although there are known exotic  symplectic $4$--manifolds in the homeomorphism classes of $ \CP \# m \CPb$ for as small as $m= 2$ \cite{AP2blowup} (also see \cite{FSpinwheels}),  to date,  the smallest  symplectic exotic $4$--manifold produced using  the rational blowdown operation is an exotic $\CP\#5\CPb$ \cite{parkparkshin3} (also see \cite{parkstipsiczszabo,  FSdoublenode} for examples that are not known to be symplectic).  Despite the first breakthroughs in constructions of small exotic rational surfaces all being via rational blowdowns \cite{park, stipsiczszabo, parkstipsiczszabo, FSdoublenode},  and a massive amount of literature on applications of rational blowdowns and their generalizations  over two decades,  the following question is still open:

\begin{question}
Is there an exotic $\CP \# m \CPb$ with $m < 5$ that can be obtained from a standard rational surface via rational blowdowns? 
If so,  what is the smallest such $m$?
\end{question}

In principle, by expanding our view to higher genus Lefschetz fibrations, we open the door to finding larger configurations of spheres that can be rationally blown down, leading to  possibly smaller exotic 4-manifolds.  Adapting this approach, we tackled this problem and we discovered many new constructions of examples right on the border.  To illustrate,  we will present a relatively simple construction, where we will start  with a particular genus--$3$ Lefschetz fibration on $\CP\#17\CPb$.  We construct this Lefschetz fibration by refactoring the monodromy associated to the hyperelliptic Lefschetz fibration on $\CP\#17\CPb$ by clustering nodes; see Lemma~\ref{lem:genus=3}.  After some blow-ups and one blow-down, we find a fairly simple configuration of spheres embedded in $\CP\#32\CPb$ that can be rationally blown down.   We then prove the following.

\begin{theorem}
There exists a minimal symplectic exotic $\CP\#5\CPb$ obtained by rationally blowing down a blow-up of a genus--$3$ Lefschetz fibration on a standard rational surface.
%$\CP\#32\CPb$.
\label{thm:exotic}
\end{theorem}

Peculiarly,  yet  another reason for us to choose this example is that only a slight potential improvement of the positive factorization we employed in the proof of this theorem would result in an exotic $\CP \# 4\CPb$; see Remark~\ref{improve}.  In fact, we expect the answer to the existence part of Question~2 to be yes,  despite our own efforts thus far falling short, noting that our investigation has mostly focused on configurations in low genera fibrations. 

\medskip
Our paper is organized as follows. In Section \ref{sec:monodromy}, we construct positive factorizations for Lefschetz fibrations with clusters of nodes; see Lemmas~\ref{lem:genus=g} and~\ref{lem:genus=3}.  We use these factorizations to then prove Theorem \ref{thm:gplane} and Proposition \ref{newishhomotopye(n)} in Section \ref{sec:geography},  and Theorem \ref{thm:exotic} in Section \ref{sec:exotic}.  Throughout the article, any $4$--manifold we consider will be compact,  connected,  smooth, oriented and without boundary, unless explicitly stated otherwise. 

%\enlargethispage{0.3in}
\vspace{0.1in}
\noindent \textit{Acknowledgements. } R.I.B. was  supported by the  NSF grant  DMS-2005327.   M. K. thanks UMass Amherst for their generous support and wonderful research environment during this project.

% ============================================================
\section{Clustering nodes in higher genera Lefschetz fibrations}\label{sec:monodromy} 

Let $\Sigma_g^m$ denote a compact connected orientable surface of genus $g$ with $m$ boundary components and let $\M(\Sigma_g^m)$ be its \textit{mapping class group}, the group of isotopy classes of orientation-preserving diffeomorphisms of $\Sigma_g^m$ that restrict to identity on the boundary $\partial \Sigma_g^m$.  We write $\Sigma_g$ for $\Sigma_g^0$.  For a simple closed curve $c$ on $\Sigma_g^m$,  we denote the positive (right-handed) Dehn twist along $c$ by $t_c$.  The conjugation of a group element $\psi$ by $\phi$, namely $\phi \, \psi \, \phi^{-1}$,  will be written as  $\psi^\phi$.  A conjugate of a positive Dehn twist satisfies $t_c^\phi=t_{\phi(c)}$,  so it is also a positive Dehn twist, for any $\phi \in \M(\Sigma_g^m)$.  The reader should not confuse the  power $t_c^k$ of a Dehn twist with the conjugate element $t_c^\phi$,  which can be differentiated by $k$ always denoting an integer and $\phi$  a mapping class.

Let  $c_1,\ldots,c_{2g+1}$ denote the curves on $\Sigma_g$ depicted in Figure~\ref{fig:genus=g} and let us set $t_i=t_{c_i}$ for a short-hand notation.  The element
\[h=t_1t_2\cdots t_{2g}t_{2g+1}t_{2g+1}t_{2g}\cdots t_2t_1 \,, \] 
 is the hyperelliptic involution on $\Sigma_g$,  which fixes each $c_i$ setwise.   We thus obtain a positive factorization of a genus--$g$ Lefschetz fibration from $h^2=1$, which is known to be a genus--$g$ hyperelliptic Lefschetz fibration on $\CP\#(4g+5)\CPb$.

 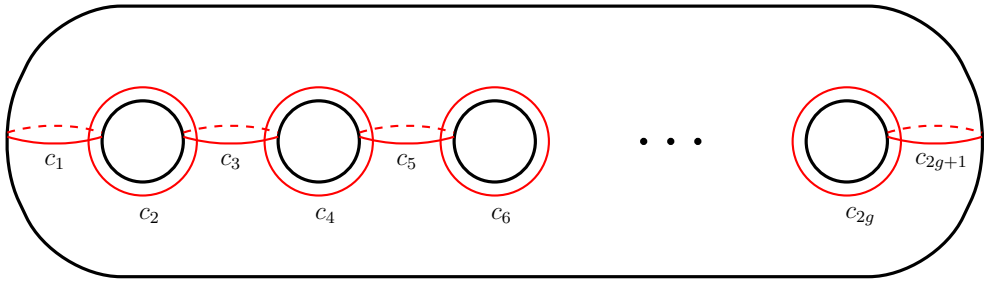
\begin{figure}[h!]
	\begin{tikzpicture}[scale=1.5]
	%%%%%% g=3 1111111
	\begin{scope} [xshift=0cm, yshift=0cm, scale=0.6]   6.15
	\draw[very thick,rounded corners=25pt] (6.5,2) --(-6.5,2) --(-7.45, 0)--(-6.5,-2) -- (6.5,-2)--(7.45,0) --cycle;
	\draw[very thick, xshift=-5.2cm] (0,0) circle [radius=0.6cm];
	\draw[very thick, xshift=-2.6cm] (0,0) circle [radius=0.6cm];
	\draw[very thick, xshift=0cm] (0,0) circle [radius=0.6cm];
	\draw[very thick, xshift=5.2cm] (0,0) circle [radius=0.6cm];
	\filldraw[very thick, xshift=2.2cm] (0,0) circle [radius=0.03cm];
	\filldraw[very thick, xshift=2.6cm] (0,0) circle [radius=0.03cm];
	\filldraw[very thick, xshift=3cm] (0,0) circle [radius=0.03cm];
	%2,4,6
	\draw[thick, red, xshift=-5.2cm] (0,0) circle [radius=0.8cm];
	\draw[thick, red, xshift=-2.6cm] (0,0) circle [radius=0.8cm];
	\draw[thick, red, xshift=0cm] (0,0) circle [radius=0.8cm];
	\draw[thick, red, xshift=5.2cm] (0,0) circle [radius=0.8cm];
	% 1
	\draw[thick, red, rounded corners=8pt, xshift=-7.8cm, yshift=0.1cm]             (0.6,-0.03) -- (0.9,-0.13)--(1.7,-0.13) --(2,-0.03);
	\draw[thick,  red, dashed, rounded corners=8pt, xshift=-7.8cm, yshift=0.1cm] (0.6,0.03) -- (0.9,0.13)--(1.7,0.13) --(2,0.03);
	% 3
	\draw[thick, red, rounded corners=8pt, xshift=-5.2cm, yshift=0.1cm]             (0.6,-0.03) -- (0.9,-0.13)--(1.7,-0.13) --(2,-0.03);
	\draw[thick,  red, dashed, rounded corners=8pt, xshift=-5.2cm, yshift=0.1cm] (0.6,0.03) -- (0.9,0.13)--(1.7,0.13) --(2,0.03);
	% 5
	\draw[thick, red, rounded corners=8pt, xshift=-2.6cm, yshift=0.1cm]             (0.6,-0.03) -- (0.9,-0.13)--(1.7,-0.13) --(2,-0.03);
	\draw[thick,  red, dashed, rounded corners=8pt, xshift=-2.6cm, yshift=0.1cm] (0.6,0.03) -- (0.9,0.13)--(1.7,0.13) --(2,0.03);
	% 2g+1
	\draw[thick, red, rounded corners=8pt, xshift=5.2cm, yshift=0.1cm]             (0.6,-0.03) -- (0.9,-0.13)--(1.7,-0.13) --(2,-0.03);
	\draw[thick,  red, dashed, rounded corners=8pt, xshift=5.2cm, yshift=0.1cm] (0.6,0.03) -- (0.9,0.13)--(1.7,0.13) --(2,0.03);
	%% labels
	\node[scale=0.8] at (-6.5,-0.3) {$c_1$};
	\node[scale=0.8] at (-3.9,-0.3) {$c_3$};
	\node[scale=0.8] at (-1.3,-0.3) {$c_5$};
	\node[scale=0.8] at (6.6,-0.3) {$c_{2g+1}$};
	\node[scale=0.8] at (-5.1,-1.1) {$c_2$};
	\node[scale=0.8] at (-2.5,-1.1) {$c_4$};
	\node[scale=0.8] at (0.1,-1.1) {$c_6$};
	\node[scale=0.8] at (5.4,-1.1) {$c_{2g}$};
	\end{scope}
	\end{tikzpicture}
	\caption{Curves on $\Sigma_g$. } \label{fig:genus=g}
\end{figure}

\begin{lemma} 
For any  integers $p, q \geq 0$ and $g \geq 1$ satisfying $p+q=4g+4$,  there are positive factorizations in $\M(\Sigma_g)$ of the form
\[t_1^{p} \cdot t_3^{q} \cdot D_{p,q, g}   =1 \, , \]
where $D_{p,q, g}$ is a product of $4g$  positive Dehn twists, and $t_i=t_{c_i}$ for $c_i$ as  in Figure~\ref{fig:genus=g}.
\label{lem:genus=g}
\end{lemma}

\begin{proof}
The statement will follow from a sequence of manipulations  in $\M(\Sigma_g)$:
\begin{alignat*}{10}
1 &=  h^2\\
&=  h\cdot (t_1t_2\cdots t_{{2g+1}} t_{{2g+1}}\cdots t_2t_1 )   \\
&=  t_1t_2\cdots t_{{2g+1}}\cdot h \cdot  t_{{2g+1}}\cdots t_2t_1 \, , \\
&=  t_1t_2\cdots t_{{2g+1}}\cdot  ( t_1t_2\cdots t_{{2g+1}} t_{{2g+1}}\cdots t_2t_1) \cdot  t_{{2g+1}}\cdots t_2t_1  \, , \\
\intertext{which follows from  $h$ fixing every $c_i$ and therefore commuting with every $t_{c_i}$. Since $t_1$ commutes with $t_j$ for all $j>2$, the two $t_1$ elements on the ends of the factorization $(t_1t_2\cdots t_{{2g+1}} t_{{2g+1}}\cdots t_2t_1)$ in the middle of the last equation above can be moved outward to the third and third-to-last positions, obtaining the equation}  
&=t_1 (t_2t_1) t_3 \cdots t_{{2g+1}}\cdot  (t_2\cdots t_{{2g+1}} t_{{2g+1}}\cdots t_2) \cdot  t_{{2g+1}}\cdots t_3 (t_1t_2)t_1\\  
\intertext{Iterating the same step for each $t_i$, for $i=1, \ldots, 2g$,  yields the next equality:}
&=t_1  (t_2t_1) (t_3t_2)\cdots (t_{{2g+1}}t_{{2g}}) t_{{2g+1}}   \cdot  t_{{2g+1}}   (t_{{2g}} t_{{2g+1}})\cdots (t_2t_3)(t_1t_2) t_1\\
\intertext{After a cyclic permutation we can bring $t_1$ on the left to the far right, recalling that the whole product is equal to the identity.  Applying a sequence of braid relations  $t_i t_{i+1} t_i = t_{i+1} t_i  t_{i+1}$ we can then carry each $t_{2g+1}$ in the middle of the factorization all the way to the right to derive
}  
&=  (t_2t_1) (t_3t_2)\cdots (t_{{2g+1}}t_{{2g}}) t_{{2g+1}}   \cdot  t_{{2g+1}}   (t_{{2g}} t_{{2g+1}})\cdots (t_2t_3)(t_1t_2)\cdot  t_1^2\\
&=   (t_2t_1) (t_3t_2)\cdots (t_{{2g+1}}t_{{2g}})  \cdot   (t_{{2g}} t_{{2g+1}})\cdots (t_2t_3)(t_1t_2) \cdot t_1^4\\
\intertext{Now,  conjugating $t_1$ with $t_2$, we can move  $t_2$ on the far left to the second position to obtain}
&=   t_1^{t_2} t_2 (t_3t_2)\cdots (t_{{2g+1}}t_{{2g}}) \cdot    (t_{{2g}} t_{{2g+1}})\cdots (t_2t_3)(t_1t_2)\cdot  t_1^4\\
\intertext{Using braid relations again,  we then carry this $t_2$ first to the very center as $t_{2g+1}$ and then to far right as $t_1$  as follows } 
&=   t_1^{t_2}(t_3t_2)\cdots (t_{{2g+1}}t_{{2g}}) \cdot   t_{{2g+1}} \cdot  (t_{{2g}} t_{{2g+1}})\cdots (t_2t_3)(t_1t_2)\cdot  t_1^4\\
&=   t_1^{t_2} (t_3t_2)\cdots (t_{{2g+1}}t_{{2g}}) \cdot    (t_{{2g}} t_{{2g+1}})\cdots (t_2t_3)(t_1t_2)\cdot  t_1^5\\
\intertext{Repeating the same steps this time for the leftmost $t_3$ we obtain}
&=   t_1^{t_2} t_2^{t_3}(t_4t_3)\cdots (t_{{2g+1}}t_{{2g}})    \cdot    (t_{{2g}} t_{{2g+1}})\cdots (t_2t_3)(t_1t_2)\cdot  t_1^6\\
\intertext{Iterate the same sequence of modifications for each $t_{i+1}$ that appears in a pair $(t_{i+1} t_i)$ (each grouped  in separate parentheses) on the left half of the factorization,  for $i=1, 2, \ldots, 2g+1$, in this order.  Then repeat for each $t_i$ that appear in pairs $(t_{i} t_{i+1})$ on the right half of the factorization, for $i=2g, \ldots, 2, 1$, in this order. We obtain}
&=  t_1^{t_2} t_2^{t_3} t_3^{t_4}\cdots t_{{2g}}^{t_{2g+1}}  \cdot    t_{2g+1}^{t_{{2g}}} \cdots t_3^{t_2} t_2^{t_1}\cdot  t_1^{4g+4}  \\
\intertext{Finally we observe that $(t_3^{t_2} t_2^{t_1}) t_1= t_3 (t_3^{t_2} t_2^{t_1})$ by braid relations.  
For  $p+q=4g+4$, repeating this $q$ times we get}
&=  t_1^{t_2} t_2^{t_3} t_3^{t_4}\cdots t_{{2g}}^{t_{2g+1}}  \cdot    t_{2g+1}^{t_{{2g}}} \cdots  t_4^{t_3}\cdot  t_3^{q} \cdot  (t_3^{t_2} t_2^{t_1})\cdot  t_1^{p}\\
&=  t_1^{t_2} t_2^{t_3} t_3^{t_4}\cdots t_{{2g}}^{t_{2g+1}}  \cdot    t_{2g+1}^{t_{{2g}}} \cdots  t_4^{t_3}\cdot    t_3^{t_2 t_3^{q}} t_2^{t_1t_3^{q}}   \cdot  t_3^{q}\cdot t_1^{p} \\
&=  t_1^{p}\cdot t_3^{q}  \cdot D_{p, q, g}
\end{alignat*}
where the very last equation is obtained by cyclic permutation and also that $t_1^p$  and $t_3^q$ commute with each other. 
\end{proof}

\smallskip
\begin{remark}
We expect that  the positive factorizations \,$t_1^{4g+4} D =1$ in $\M(\Sigma_g)$ we obtained above to be optimal for any $g \geq 2$,  that is,   one cannot perturb the given Lefschetz fibration on $X:=\CP \# \, (4g+5) \CPb$ to cluster identical nodes (i.e.  with isotopic vanishing cycles) on a single singular fiber.  In fact,  a cluster corresponding to the $t_1^{k}$ factor yields a chain of $k-1$ $(-2)$-spheres which span a negative-definite subspace of $V$ of $H_2(X; \R)$ for the intersection form $Q_X$.  Moreover, for the symplectic regular fiber $F$,  we have $[F] \neq 0$ in  $H_2(X; \R)$ and $[F] \in V^{\perp}$,  the  orthogonal complement of $V$ with respect to $Q_X$.  Because $Q_X|_{V^{\perp}}$ is nondegenerate, there is an additional class in
 $V^{\perp}$ with negative square.  Hence $4g+4=b^-(X) \geq k$, which demonstrates that we are at most one off from clustering the maximal number of nodes.   Generally, we conjecture that for a fiber sum indecomposable (see e.g. \cite{baykurfibersum} for the definition) genus $g \geq 2$ Lefschetz fibration,  the maximal number of  nodes one can cluster like this is $4g+4$,  realized by the fibrations we get on the rational surface $\CP \# \, (4g+5) \CPb$.  Note that this number is $4g+5$ when $g=1$.
\end{remark}

We will now construct a particular positive factorization for a genus--$3$ Lefschetz fibration.  In this case, for  the rational blowdown configurations we desire to get,  it will be essential to identify some sections as well.  We will thus produce a positive factorization of the boundary multi-twist $t_{\delta_1} t_{\delta_2}$ in $\M(\Sigma_{3}^2)$.  (See e.g. \cite{baykurkorkmaz} for how boundary twists yield sections.)
%Let $a,b,c_1,\ldots, c_7$ denote the curves shown in Figure \ref{fig:genus=3}.

 \begin{figure}[h!]
\begin{tikzpicture}[scale=1.5]
%%%%%% g=3 1111111
\begin{scope} [xshift=0cm, yshift=0cm, scale=0.6]   6.15
 \draw[very thick,rounded corners=25pt] (-4.8,-2) -- (3.9,-2)--(4.85,0) -- (3.9,2) --(-4.8,2)  ;
 \draw[very thick,  rounded corners=12pt] (-4.8,1) -- (-4.4,0) -- (-4.8,-1)  ;
 \draw[very thick, xshift=-2.6cm] (0,0) circle [radius=0.6cm];
 \draw[very thick, xshift=0cm] (0,0) circle [radius=0.6cm];
 \draw[very thick, xshift=2.6cm] (0,0) circle [radius=0.6cm];
 \draw[very thick]  (-4.8,1.5) ellipse (0.2cm and 0.5cm);
\draw[very thick]  (-4.8,-1.5) ellipse (0.2cm and 0.5cm);

\draw[thick, red, xshift=-2.6cm] (0,0) circle [radius=0.8cm];
\draw[thick, red, xshift=0cm] (0,0) circle [radius=0.8cm];
\draw[thick, red, xshift=2.6cm] (0,0) circle [radius=0.8cm];
% 1
  \draw[thick, red, rounded corners=6pt, xshift=-5.1cm, yshift=0.1cm]             (0.6,-0.03) -- (0.9,-0.13)--(1.7,-0.13) --(1.9,-0.03);
 \draw[thick,  red, dashed, rounded corners=6pt, xshift=-5.1cm, yshift=0.1cm] (0.6,0.03) -- (0.9,0.13)--(1.7,0.13) --(1.9,0.03);
% 3
  \draw[thick, red, rounded corners=6pt, xshift=-2.6cm, yshift=0.1cm]             (0.6,-0.03) -- (0.9,-0.13)--(1.7,-0.13) --(2,-0.03);
 \draw[thick,  red, dashed, rounded corners=6pt, xshift=-2.6cm, yshift=0.1cm] (0.6,0.03) -- (0.9,0.13)--(1.7,0.13) --(2,0.03);
% 5
  \draw[thick, red, rounded corners=6pt, xshift=0cm, yshift=0.1cm]             (0.6,-0.03) -- (0.9,-0.13)--(1.7,-0.13) --(2,-0.03);
 \draw[thick,  red, dashed, rounded corners=6pt, xshift=0cm, yshift=0.1cm] (0.6,0.03) -- (0.9,0.13)--(1.7,0.13) --(2,0.03);
% 7
  \draw[thick, red, rounded corners=6pt, xshift=2.6cm, yshift=0.1cm]             (0.6,-0.03) -- (0.9,-0.13)--(1.7,-0.13) --(2,-0.03);
 \draw[thick,  red, dashed, rounded corners=6pt, xshift=2.6cm, yshift=0.1cm] (0.6,0.03) -- (0.9,0.13)--(1.7,0.13) --(2,0.03);

 \draw[thick,  red, rounded corners=6pt, xshift=-2.7cm ] (-0.03,0.6) -- (-0.13, 1)--(-0.13,1.6) --(-0.03,2);
\draw[thick,  red, dashed, rounded corners=6pt, xshift=-2.7cm ] (0.03,0.6) -- (0.13, 1)--(0.13,1.6) --(0.03,2);
 \draw[thick,  red, rounded corners=6pt, xshift=-2.7cm ] (-0.03,-0.6) -- (-0.13, -1)--(-0.13,-1.6) --(-0.03,-2);
\draw[thick,  red, dashed, rounded corners=6pt, xshift=-2.7cm ] (0.03,-0.6) -- (0.13, -1)--(0.13,-1.6) --(0.03,-2);

 \node[scale=.8] at (-3.9,-0.3) {$c_1$};
 \node[scale=0.8] at (-1.3,-0.3) {$c_3$};
 \node[scale=0.8] at (1.3,-0.3) {$c_5$};
 \node[scale=0.8] at (3.9,-0.3) {$c_7$};
 \node[scale=0.8] at (-1.8,-0.8) {$c_2$};
 \node[scale=0.8] at (0.1,-1.1) {$c_4$};
 \node[scale=0.8] at (2.7,-1.1) {$c_6$};
 \node[scale=0.8] at (-2.3,1.3) {$a$};
 \node[scale=0.8] at (-2.3,-1.3) {$b$};
  \node[scale=0.8] at (-4.3,1.5) {$\delta_1$};
  \node[scale=0.8] at (-4.3,-1.5) {$\delta_2$};
\end{scope}

\end{tikzpicture}
\caption{Curves on $\Sigma_3^2$. } \label{fig:genus=3} 
\end{figure}
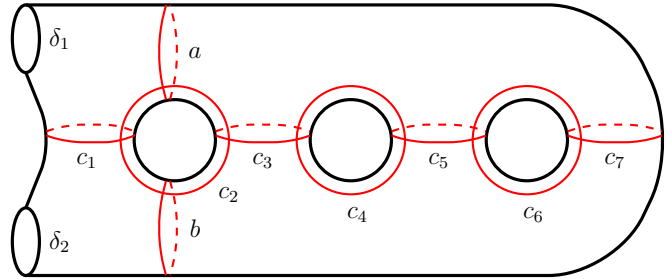

\begin{lemma}
There is a positive factorization in $\M(\Sigma_3^2)$
\[ (t_1^{14} \, t_at_b t_4 t_6) \cdot   (t_1 t_3 t_5t_7)^{ t_2^{-1} t_5 t_6^{-1}}
	\cdot D_6 = t_{\delta_1}t_{\delta_2} \,,  \]
 where $D_6$ is a product of $6$ positive Dehn twists,  $t_i=t_{c_i}$, and $c_i,  a,  b, c, \delta_j$  are as in Figure~\ref{fig:genus=3}. 
\label{lem:genus=3}
\end{lemma}

\begin{proof} The statement will follow from a sequence of manipulations  this time in $\M(\Sigma_3^2)$. We start with the well-known $7$--chain relation:
\begin{alignat*}{10}
t_{\delta_1}t_{\delta_2}
&= (t_7 \cdots t_2t_1)^8\\
&= (t_7\cdots t_2t_1)(t_7\cdots t_2t_1)(t_7\cdots t_2t_1)(t_7\cdots t_2t_1)(t_7\cdots t_2t_1)(t_7\cdots t_2t_1)(t_7\cdots t_2t_1)(t_7\cdots t_2t_1) \\
\intertext{Note that $t_i (t_7\cdots t_2t_1) = (t_7\cdots t_2t_1) t_{i+1}$ for each $i=1, \ldots, 6$.  We can thus move the $t_1$ in the first  parentheses over the next six $(t_7\cdots t_2t_1)$ factors,  so it becomes $t_7$. We then move the $t_1$ in the second  parentheses over the next five $(t_7\cdots t_2t_1)$ factors,  so it becomes $t_6$.  Repeating this for each $t_1$ in the first six parentheses, we get:
}
&= (t_7\cdots t_3t_2)^7 \cdot (t_1t_2\cdots t_6t_7 \cdot t_7t_6\cdots t_2t_1)\\
\intertext{By the same method, we can now move $t_2$ factors to arrive at the factorization}
&= (t_7t_6t_5 t_4t_3)^6 \cdot  (t_2t_3\cdots t_6t_7 \cdot t_7t_6\cdots t_3t_2)\cdot (t_1t_2\cdots t_6t_7 \cdot t_7t_6\cdots t_2t_1)\\
\intertext{Using the $5$--chain relation to substitute $(t_7t_6t_5 t_4t_3)^6$  with $t_a t_b$, and braid relations to rewrite the rest of the factorization, we obtain}
&= (t_at_b)\cdot (t_1t_2\cdots t_6t_7 \cdot  t_1t_2\cdots t_5t_6 \cdot t_6t_5\cdots t_2t_1 \cdot t_7t_6\cdots t_2t_1)\\
\intertext{We can then use commutativity relations to rearrange the factorization into}
&= (t_at_b  t_1) \cdot  (t_2t_1) (t_3t_2)(t_4t_3)(t_5t_4) (t_6t_5)(t_7t_6)\cdot (t_6t_7) (t_5t_6)(t_4t_5)(t_3t_4) (t_2t_3)(t_1t_2) t_1\\
\intertext{Now,  employing  the same argument we had in the proof of Lemma~\ref{lem:genus=g},  through a sequence of conjugations and braid relations, we can cluster $t_1$s on the right we obtain}
&= (t_at_b   t_1)  \cdot  t_1^{t_2}\cdot t_2^{t_3}\cdot t_3^{t_4}\cdot t_4^{t_5}\cdot t_5^{t_6}\cdot  t_6^{t_7}\cdot  
t_7^{t_6}\cdot  t_6^{t_5}\cdot  t_5^{t_4}\cdot  t_4^{t_3}\cdot t_3^{t_2}\cdot  t_2^{t_1}\cdot  t_1^{13}\\
\intertext{We then bring all the $t_1$ factors together by cyclic permutation of $t_1^{13}$ and the commutativity of the remaining $t_1$ factor with $t_at_b$, which yields}
&= (t_1^{14}t_at_b)   \cdot  t_1^{t_2}\cdot t_2^{t_3}\cdot t_3^{t_4}\cdot t_4^{t_5}\cdot t_5^{t_6}\cdot  t_6^{t_7}\cdot  
  t_7^{t_6}\cdot  t_6^{t_5}\cdot  t_5^{t_4}\cdot  t_4^{t_3}\cdot t_3^{t_2}\cdot  t_2^{t_1} \\
 \intertext{Next carry the $t_5^{t_4}$ factor all the way to the right (note that it commutes with $t_2^{t_1}$) to obtain}
&= (t_1^{14}t_at_b)   \cdot  t_1^{t_2}\cdot t_2^{t_3}\cdot t_3^{t_4}\cdot t_4^{t_5}\cdot t_5^{t_6}\cdot  t_6^{t_7}\cdot  
  t_7^{t_6}\cdot  t_6^{t_5}\cdot  ( t_4^{t_3} t_3^{t_2})^{{t_5^{t_4}}} 
  \cdot  t_2^{t_1} \cdot  t_5^{t_4}\\
\intertext{By cyclic permutation and commutativity of $t_5^{t_4}$ with $t_1^{14}t_a t_b$,  we can then carry the $t_5^{t_4}$ factor and place it in the first cluster to obtain}
&= (t_1^{14}t_at_b t_5^{t_4})  \cdot  t_1^{t_2}\cdot t_2^{t_3}\cdot t_3^{t_4}\cdot t_4^{t_5}\cdot t_5^{t_6}\cdot  t_6^{t_7}\cdot  
  t_7^{t_6}\cdot  t_6^{t_5}\cdot  ( t_4^{t_3} t_3^{t_2})^{{t_5^{t_4}}} \cdot  t_2^{t_1} \\
 \intertext{By moving away the Dehn twists in between by conjugations, we create a second cluster as follows}
&= (t_1^{14}t_at_b t_5^{t_4})  \cdot   (t_2^{t_3}  t_5^{t_6}   t_6^{t_7}t_6^{t_5}t_2^{t_1})
	\cdot D'_6 \\
\intertext{where $D'_6$ denotes the product of the remaining six positive Dehn twists. Here $t_2^{t_3}$ and $t_5^{t_6}$ commute, so we can move the latter into the first cluster to get}
&= (t_1^{14}t_at_b t_5^{t_4} t_5^{t_6}) \cdot   (t_2^{t_1} t_2^{t_3}     t_6^{t_5}t_6^{t_7}) 
	\cdot D'_6 \\
\intertext{In the next step, we  repeatedly make use of the following observation: if $x$ and $y$ intersect at one point, then the braid relation $t_x t_y t_x = t_y t_x t_y$ implies that $t_y^{-1}t_x t_y=  t_x t_y t_x^{-1}$, and in turn, we have $t_x^{t_y^{-1}}= t_y^{t_x}$.  So we can rewrite the last factorization as}
&= (t_1^{14}t_at_b t_4^{t_5^{-1}} t_6^{t_5^{-1}}) \cdot   (t_1^{t_2^{-1}}  t_3^{t_2^{-1} } t_5^{ t_6^{-1}}t_7^{ t_6^{-1}})
	\cdot D'_6 \\
\intertext{Since $t_1^{14} t_a t_b$ commutes with $t_5^{-1}$,  $t_1 t_3$ commutes with $t_6^{-1}$, and $t_5 t_7$ commutes with $t_2^{-1}$, we obtain}	
&= (t_1^{14}t_at_b t_4 t_6)^{t_5^{-1}} \cdot   (t_1 t_3 t_5t_7)^{ t_2^{-1}  t_6^{-1}}
	\cdot D'_6 \  \\
\intertext{We can conjugate this expression with $t_5$ to obtain:} 
&= (t_1^{14}t_at_b t_4 t_6) \cdot   (t_1 t_3 t_5t_7)^{ t_5 t_2^{-1}  t_6^{-1}}
	\cdot  D_6.
\end{alignat*}
Finally, noting  that $t_5$ commutes with $t_2^{-1}$, the factorization in the statement of the lemma follows.
\end{proof}

%\begin{eqnarray*}
%\delta_1 \delta_2
%&=&(7654321)^8\\
%&=&(765432)^7(12345677654321)\\
%&=&(76543)^6(234567765432)(12345677654321)\\
%&=&ab (1234567(123456654321)7654321)\\
%&=&ab 1 (21)(32)(43)(54)(65)(76) (67)(56)(45)(34)(23)(12)1)\\
%&=& ab1 \cdot 1^2 2^3 3^4 4^5 5^6 6^7 7^6 6^5 5^4 4^3 3^2 2^1 \cdot 1111111111111\\
%&=& ab1 \cdot 1^2 2^3 3^4 4^5 5^6 6^7 7^6 6^5 5^4 4^3 3^2 2^1 \cdot 1^{(13)}\\
%&=& (ab1^{(14)}) \cdot 1^2 2^3 3^4 4^5 5^6 6^7 7^6 6^5 5^4 4^3 3^2 2^1 \\
%&=& (ab1^{(14)}) \cdot  2^3 5^6 6^7 6^5 2^1 \cdot 1^*  3^* 4^*  7^*  5^4 4^* 3^*  \\
%&=& (ab1^{(14)})\cdot   2^3 5^6 6^7 6^5 2^1 \cdot 1^*  3^* 4^*  7^* 4^* 3^*   5^4 \\
%&=& (ab1^{(14)}   5^4 5^6) \cdot  (2^1 2^36^7 6^5 ) \cdot 1^*  3^* 4^*  7^* 4^* 3^*  \\
%&=& (ab1^{(14)}   4^{\bar 5} 6^{\bar 5} ) \cdot  (1^{\bar 2}3^{\bar 2}  5^{\bar 6} 7^{\bar 6} ) \cdot 1^*  3^* 4^*  7^* 4^* 3^*  \\ 
%&=& (ab1^{(14)}   4 6)^{\bar 5} \cdot  (13  57)^{\bar 2 \bar 6}  \cdot 1^*  3^* 4^*  7^* 4^* 3^* , \ (\mbox{conjugate with) }5\\ 
%&=& (ab1^{(14)}   4 6) \cdot (13  57)^{\bar 2  5\bar 6}  \cdot D_6. 
%\end{eqnarray*}
%Here, $D_6$ is a product of $6$ positive Dehn twists.

% ============================================================

\section{Geography of minimal symplectic Lefschetz fibrations}\label{sec:geography}

Consider the genus $g \geq 2$ Lefschetz fibration on $\CP\# (4g+5)\CPb$ corresponding to the positive factorization $t_1^{2g+2}\cdot t_3^{2g+2} \cdot D_g =1 $ in $\M(\Sigma_g)$  derived in Lemma \ref{lem:genus=g}, where $D_g:=D_{2g+2, 2g+2, g}$.  Performing an untwisted fiber sum of this fibration  with itself is known to yield the complex surface $E(g+1)$.  (To see it note that this fibration on $\CP\# (4g+5)\CPb$, after a small perturbation,  is isomorphic  to a holomorphic fibration and the untwisted fiber sum of two copies can be realized as a double branched cover ramified along two copies of the regular fiber.)

Instead, we are interested in the result of a twisted fiber sum, which yields a homotopy $E(g+1)$.  After a global conjugation, the monodromy factorization of the above fibration can be written as
\[ A_{g} \cdot t_{a_1}^{2g+2}\cdot t_{a_3}^{2g+2} =1 \, ,\text{ and  also as } t_{b_1}^{2g+2}\cdot t_{b_3}^{2g+2}\cdot B_{g}=1 \text{ in } \M(\Sigma_g) \, ,  \]
where $a_1$, $a_3, b_1, b_3$ are the curves shown in Figure \ref{fig:twistedfibersum} and  $A_{g}$ and $B_g$ are each products of $4g$ positive Dehn twists.  Then, let $(Z_{g+1},  f_{g+1})$ denote the Lefschetz fibration with monodromy factorization 
\[A_{g} \cdot t_{a_1}^{2g+2}\cdot t_{a_3}^{2g+2} \cdot t_{b_1}^{2g+2}\cdot t_{b_3}^{2g+2}\cdot B_{g} =1  \text{ in } \M(\Sigma_g) \, ,  \]
which is a \textit{twisted} fiber sum of two copies of the fibration we had on  $\CP\# (4g+5)\CPb$.  

We claim that $Z_{g+1}$ is simply-connected.  One way to see this is the following: our Lefschetz fibration on $\CP\# (4g+5)\CPb$,  after a small perturbation,  becomes isomorphic to the standard holomorphic Lefschetz fibration,  which has many sections; see e.g. \cite{tanaka}.  Therefore, by Seifert Van-Kampen,  the complement of the regular fiber in the simply-connected \mbox{$4$-manifold} $\CP\# (4g+5)\CPb$ is necessarily simply-connected, and in turn,  the twisted fiber sum $Z_{g+1}$ obtained by gluing the two complements is also simply-connected.

%Recall that the monodromy of the Lefschetz fibration on $\CP\# (4g+5)\CPb$,  up to global conjugation and Hurwitz moves,  is equivalent to a product of the Dehn twists $t_i$.  If we take out an open disk from the top half of the surface $\Sigma_g$ in Figure~\ref{fig:genus=g},   away from all the Dehn twist curves $c_i$,  one can see that the positive factorization lifts to $\Sigma_g^1$

Easy Euler characteristic and signature calculations show  that we have  $\chi(Z_{g+1})=12g+12=\chi(E(g+1))$ and $\sigma(Z_{g+1})=-8g-8=\sigma(E(g+1))$.   Moreover,  the fiber sum $Z_{g+1}$ is 
minimal by  \cite{usherminimal} (also see \cite{baykurfibersum}).
When $g$ is even, $Z_{g+1}$ is certainly not spin,  for instance  by Rokhlin's theorem.  On the other hand, when $g$ is odd,  with some extra care,  one can show that $Z_{g+1}$ is spin using the spin criterion for monodromy factorizations of Lefschetz fibrations in \cite{stipsiczspin, baykurhamada}.  (Here the twisting matters;  one may for instance take the second positive factorization in the fiber sum as the image of the first factorization under the obvious hyperelliptic involution, followed by a cyclic permutation.) Hence by Freedman \cite{freedman},  $Z_{g+1}$ is homeomorphic to $E(g+1)$ for all $g \geq 2$.

\begin{figure}[h!]
	\centering
	\begin{tikzpicture}[scale=1.5]
	
	\begin{scope} [xshift=0cm, yshift=0cm, scale=0.6]   6.15
	\draw[very thick,rounded corners=25pt] (2.7,2) --(-6.5,2) --(-7.45, 0)--(-6.5,-2) -- (2.7,-2)--(3.65,0) --cycle;
	\draw[very thick, xshift=-5.2cm] (0,0) circle [radius=0.6cm];
	\draw[very thick, xshift=-2.6cm] (0,0) circle [radius=0.6cm];
	\draw[very thick, xshift=1.4cm] (0,0) circle [radius=0.6cm];
	\filldraw[very thick, xshift=-1cm] (0,0) circle [radius=0.03cm];
	\filldraw[very thick, xshift=-.6cm] (0,0) circle [radius=0.03cm];
	\filldraw[very thick, xshift=-.2cm] (0,0) circle [radius=0.03cm];
	%a1
	\draw[thick, red, dashed, rounded corners=8pt, xshift=-5.2cm, yshift=0cm]             (0.03,-0.6) -- (0.13,-0.9)--(0.13,-1.7) --(0.03,-2);
	\draw[thick,  red,  rounded corners=8pt, xshift=-5.2cm, yshift=0cm]         (-0.03,-0.6) -- (-0.13,-0.9)--(-0.13,-1.7) --(-0.03,-2);
	% a3
	\draw[thick, red, dashed, rounded corners=8pt, xshift=-2.5cm, yshift=0cm]             (0.03,-0.6) -- (0.13,-0.9)--(0.13,-1.7) --(0.03,-2);
	\draw[thick,  red, rounded corners=8pt, xshift=-2.5cm, yshift=0cm]         (-0.03,-0.6) -- (-0.13,-0.9)--(-0.13,-1.7) --(-0.03,-2);
		%b1
	\draw[thick, blue, dashed, rounded corners=8pt, xshift=-5.2cm, yshift=0cm]             (0.03,0.6) -- (0.13,0.9)--(0.13,1.7) --(0.03,2);
	\draw[thick,  blue,  rounded corners=8pt, xshift=-5.2cm, yshift=0cm]         (-0.03,0.6) -- (-0.13,0.9)--(-0.13,1.7) --(-0.03,2);
	% b3
	\draw[thick, blue, dashed, rounded corners=8pt, xshift=-2.5cm, yshift=0cm]             (0.03,0.6) -- (0.13,0.9)--(0.13,1.7) --(0.03,2);
	\draw[thick,  blue, rounded corners=8pt, xshift=-2.5cm, yshift=0cm]         (-0.03,0.6) -- (-0.13,0.9)--(-0.13,1.7) --(-0.03,2);
	
	\node[scale=0.8] at (-5.7,-1.3) {$a_1$};
	\node[scale=0.8] at (-3,-1.3) {$a_3$};
	\node[scale=0.8] at (-5.7,1.2) {$b_1$};
	\node[scale=0.8] at (-3,1.2) {$b_3$};
	\end{scope}
	\end{tikzpicture}
\caption{Lantern curves in the twisted fiber sum $(Z_{g+1}, f_{g+1})$. }
 \label{fig:twistedfibersum}
\end{figure}
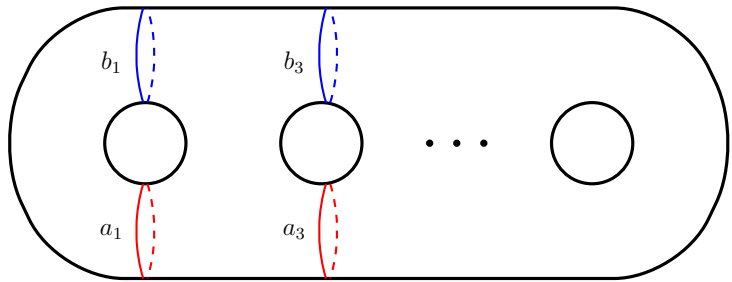

We will next show that $Z_{g+1}$ is not diffeomorphic to $E(g+1)$. A natural followup question is whether it is diffeomorphic to a well-known homotopy elliptic surface in the literature.  Recall that a homotopy $E(n)$ obtained by logarithmic transformations  for relatively prime $p,q\ge 1$ is denoted by $E(n)_{p,q}$,  and the one obtained by knot surgery along a generic elliptic fiber is denoted by $E(n)_K$.   While any $E(n)_{p,q}$ is  symplectic,  $E(n)_K$ is symplectic if and only if $K$ is fibered, and when that is the case, $E(n)_K$ admits a genus $2l+(n-1)$ Lefschetz fibration \cite{fintushelstern, FSknotsurgery, fintushelsternfamilies}.  

\begin{prop} \label{newishhomotopye(n)}
	If $g>3$, then $Z_{g+1}$ is not diffeomorphic to $E(g+1)_{p,q}$ for any $p,q\ge 1$ or $E(g+1)_K$ for any knot $K$. In particular,  $Z_{g+1}$ is not diffeomorphic to $E(g+1)$. 
\end{prop}

\begin{proof}
Let $K$ denote the canonical class of $Z_g$ and let $S$ denote an embedded symplectic sphere with square $-4$. By the adjunction equality, $\langle K, S\rangle = -\chi(S)-[S]^2=2$.  Suppose $Z_g=E(g+1)_{p,q}$. Then $K=((g+1)pq-p-q)f$, where $f$ is a primitive class such that $pqf=F$ is the elliptic fiber class in $E(g+1)$. Thus $2=\langle K, S\rangle = ((g+1)pq-p-q)f\cdot S$. Thus $(g+1)pq-p-q\in\{1,2\}$. Now since $pq-p-q\ge -1$ and $pq \geq 1$ for all integers $p,q\ge 1$, we necessarily have that $g\le 3$. 
	
Next suppose $Z_{g+1}=E(g+1)_K$.  If $K$ is not a fibered knot, then $E(g+1)_K$ does  not admit a symplectic structure \cite{FSknotsurgery}, so $Z_{g_1}$ could not be diffeomorphic to it.  Assume $K$ is a fibered knot of genus $l \geq 0$.  We first claim that the canonical class must be of the form $K=(2l+g-1)F$, where $F$ is the elliptic fiber class of $E(g+1)$ and $l$ is the genus of $K$. Let $\Sigma$ denote a generic fiber of the symplectic Lefschetz fibration on $E(g+1)_K$, which has genus $2l+g$ and square 0.  From the construction,  we know that $F$ is a bisection of this fibration and $F \cdot \Sigma =2$.  By the adjunction equality, $\langle K, \Sigma\rangle = 2(2l+g)-2=4l+2g-2$. By \cite{FSknotsurgery}, $\mathcal{SW}_{E(g+1)_K}=\mathcal{SW}_{E(g+1)}\cdot\Delta_K(t^2)$, where $t=\exp{(F)}$ and $\Delta_K$ is the symmetrized Alexander polynomial of $K$; consequently, the canonical class $K$ must be of the form $mF$ for some integer $m$. Thus we have $4l+2g-2=m \langle F, \Sigma\rangle = 2m$ and so $m=2l+g-1$, as claimed. Finally, $2=\langle K, S\rangle = (2l+g-1)F\cdot S$,  which implies $g\le 3$.

If we take $p=q=1$ for $E(g+1)_{p,q}$ or $K$ the unknot for $E(g+1)_K$,  we get back the standard $E(g+1)$.  So the arguments above show in particular  that $Z_{g+1}$ is not diffeomorphic to $E(g+1)$. 
\end{proof}

The monodromy factorization of $(Z_{g+1},  f_{g+1})$ can be rewritten as
\begin{equation} \label{blowdownmonodromy}
 A_{g} \cdot (t_{a_1} t_{a_3}  t_{b_1}t_{b_3})^{2g+2}\cdot B_{g} =1 
\text{ in } \M(\Sigma_g) \, , 
\end{equation}
where $a_1, a_3, b_1, b_3$ can be seen to create a lantern configuration; see Figure~\ref{fig:twistedfibersum}.  Equivalently,  clustering the corresponding four nodes on the same singular fiber we get a symplectic $(-4)$-sphere contained in the regular fiber.  
By \cite{endogurtas},  lantern substitution along this configuration will provide a new positive factorization for a symplectic Lefschetz fibration that would be obtained from the former by rationally blowing down the corresponding $(-4)$-sphere.  As the above monodromy factorization shows, there are $2g+2$ of these $(-4)$--spheres in $Z_{g+1}$. 

Set $a=g-1$.  For each pair of  integers $(a,b)$ with $a \geq 3$ and $0<b \leq 2a$,  let $(Z_{a,b}, f_{a,b})$ denote the symplectic Lefschetz fibration obtained by rationally blowing down $b$ of the above $(-4)$--spheres in $(Z_a, f_a)$.  All $Z_{a,b}$ are minimal by   \cite{dorfmeisterminimal} and we have $\chi_h(Z_{a,b})= \chi_h(Z_{a})= a$ and $c_1^2(Z_{a,b})=c_1^2(Z_{a}) +b = b$. 
Thus the collection $$\{Z_{a,b}\text{ }|\text{ } a\ge 3\text{ and } 0 < b\le 2a\}$$ of minimal symplectic 4-manifolds fills the region of the geography plane $\mathcal{R}$ in Figure \ref{fig:gplane}. 

The claim that $Z_{a,b}$ is non-spin follows from the following observation: Each time we apply a lantern substitution along a \,$t_{a_1} t_{a_3}  t_{b_1}t_{b_3}$\, subfactor in the monodromy factorization,  we get a new vanishing cycle that separates the pairs $\{a_1, b_1\}$ and $\{a_3, b_3\}$. This is a separating curve, and after a small perturbation of the fibration, we get a reducible fiber that contains a genus--$1$ (and a genus $g-1$) fiber component with self-intersection $-1$. Hence, the intersection form of the ambient manifold has to be odd. 

It remains to show that $Z_{a,b}$ are all simply-connected.  This is fairly easy to when $b<2a$.  Recall that if $\{v_j\}$ are the vanishing cycles of a genus--$g$ Lefschetz fibration $(X,f)$ with a section,  the fundamental group of $X$ is equal to $\pi_1(\Sigma_g) \, / N$ where $N$ is the subgroup of $\pi_1(\Sigma_g)$ normally generated by $\{v_j\}$.   Therefore,  the vanishing cycles of the Lefschetz fibration we started with on the simply-connected $4$--manifold $\CP \# \, (4g+5) \CPb$ should normally generate all of $\pi_1(\Sigma_g)$.   In turn,  the vanishing cycles in the factor $A_g  t_{a_1} t_{a_3}$ (or  in $t_{b_1} t_{b_3} B_g$) would alone normally generate  all of $\pi_1(\Sigma_g)$.  But observe that until we make substitutions along all  $t_{a_1} t_{a_3} t_{b_1} t_{b_3}$  factors,  the original collection of the vanishing cycles for the Lefschetz fibration on $\CP \# \, (4g+5) \CPb$  (obviously after a global conjugation) are still part of the monodromy,  which implies that $Z_{a,b}$ we get  for $b <2a$ is  simply-connected.  For the remaining $b=2a$ case,  one can similarly check  that there are still an abundant collection of vanishing cycles in $Z_{a, 2a}$ that normally generate the fundamental group of the fiber.  (In fact,  it suffices to observe that if we remove $a_1$ and $a_3$ from the collection,  but add $c_3$, which is one of the curves we get after the lantern substitution,  the vanishing cycles in hand still kill the entire fundamental group of the fiber.)

This concludes the proof of Theorem~\ref{thm:gplane}.

\begin{remark} \label{FSexamples}
There is at least one more example we can quickly cook up in the region $\mathcal{R}$ below the Noether line. Consider the genus--$3$ Lefschetz fibration on $\CP \# 17 \CPb$ with monodromy factorization: 
\[ (t_1t_2t _3 t_4 t_5 t_5 t_6 t_7^2 t_6 t_5 t_4 t_3 t_2 t_1)^2 =1 \, \text{ in } \M(\Sigma_3) \, , \]
 where $t_i=t_{c_i}$ and $c_i$ are as shown in Figure~\ref{fig:genus=g}.  This factorization can be re-written as 
\[ (t_1 t_3 t_5 t_7)^4 D = \, \text{ in } \M(\Sigma_3) \,  , \]
where $D$ is a product of positive Dehn twists.  Applying one lantern substitution along the factor $t_1 t_3 t_5 t_7$ results in a genus--$3$ Lefschetz fibration, the total space of which can be shown to be $\CP \# 16 \CPb$.  If we now take the fiber sum of these two fibrations, we get a genus--$3$ Lefschetz fibration whose total space is a minimal \cite{usherminimal,  baykurfibersum} symplectic $4$--manifold with $\chi_h=4$ and $c_1^2=1$, violating the Noether inequality. 

The only other examples of relatively minimal Lefschetz fibrations in the literature we know of which realize some lattice points in the region below the Noether line (and above $c_1^2=0$)  are the ones described by Fintushel and Stern in  \cite{fintushelsternLFs}.  
By taking fiber sums of two different holomorphic genus--$g$ Lefschetz fibrations on rational surfaces,  they get examples  on the line $c_1^2=\chi_h-3$ for any $\chi_h > 3$ with  $\chi_h \not \equiv 0$ (mod $3$).  Provided one determines the topological invariants that are only implicitly expressed in \cite{fintushelsternLFs},  one can obtain  more examples following their construction scheme, but still with several limitations in the way of populating a large region; for example, one needs to have $g=\frac{1}{2}(p-1)(q-2)=\frac{1}{2}(p'-1)(q'-2)$ for pairs of relatively prime positive integers $p, q$ and $p',q'$ with $\{p,q\} \neq \{p',q'\}$. In particular, for every prime $g$, there is only one lattice point with $\chi_h=g+1$ realized by this construction. 

Moreover, a common aspect of all these sporadic examples is that they are fiber sums.  In contrast, we suspect that most (perhaps all) of our examples $(Z_{a,b}, f_{a,b})$ with $a > 4$ and $b>0$ are fiber sum indecomposable, and not diffeomorphic to any of these other examples. 
\end{remark}

\begin{remark}\label{nonspinexs}
Recall that $Z_{g+1}$ is a homotopy $E(g+1)$ for all $a\ge3$; in particular, $Z_{g+1}$ is spin when $g+1$ is even. However, we can populate all lattice points on the line $c_1^2=0$ (with $\chi_h\ge3$) with \textit{non-spin} minimal simply connected symplectic genus--$g$ Lefschetz fibrations as well. This can be achieved by using a different twisted fiber sum than the one used to form $Z_{g+1}$. Let $(Y_{g+1},  h_{g+1})$ denote the Lefschetz fibration with monodromy factorization 
\[A_{g} \cdot t_{a_1}^{2g+2}\cdot t_{a_3}^{2g+2} \cdot t_{1}^{2g+2}\cdot t_{3}^{2g+2}\cdot D_{g} =1  \text{ in } \M(\Sigma_g) \, ,  \]
where  $A_{g} \cdot t_{a_1}^{2g+2}\cdot t_{a_3}^{2g+2}=1$ and $t_{1}^{2g+2}\cdot t_{3}^{2g+2}\cdot D_{g} =1$ are the monodromy factorizations described previously. Rewriting this monodromy as $(t_{a_1} t_{a_3} t_3) \cdot D =1$,  where $D$ is a product of positive Dehn twists,  it is clear that there is a fiber containing a sphere with self-intersection $-3$, implying that the intersection form of $Y_{g+1}$ is odd. 

One may wonder if we can as well generate \emph{spin} symplectic Lefschetz fibrations in the region $\mathcal{R}$, and this too  is possible by varying our construction and using the \emph{spin substitution} techniques developed in \cite{baykurhamada}.  In this case,  we would instead  start with a twisted fiber sum $(Y'_{g+1}, f'_{g+1})$ on the $c_1^2=0$ line with a monodromy factorization  
\[ t_{a_5}^{4g+4}\cdot t_{b_5}^{4g+4} \cdot D'_{g} =1  \text{ in } \M(\Sigma_g) \, ,  \]
for $g \geq 3$ and odd,  where $\{a_5,  b_5\}$ is the next  pair of curves one can imagine in Figure~\ref{fig:twistedfibersum},  which cobound a subsurface $\Sigma_2^2$ of $\Sigma_g$.  One can then perform $5$-chain substitions along each $t_{a_5} t_{b_5}$ factor to create the desired examples; cf.  the proof of Theorem~A in \cite{baykurhamada}.
\end{remark}

\begin{remark} \label{superconformal}
Let $X$ be a simply-connected,  almost complex $4$--manifold of Seiberg-Witten simple type.   Fintushel and Stern conjectured that if the characteristic numbers of $X$ satisfy $0 \leq c_1^2 \leq \chi_h -3 $ then $X$ has at least $\chi_h - c_1^2-2$ Seiberg-Witten basic classes.  This conjecture was first confirmed for $4$--manifolds of \textit{superconformal simple type} in a physics paper \cite{superconformal}, and was later completely proved by Feehan and Leness in  \cite{feehanleness}.  Note that any symplectic $X$ is of Seiberg-Witten simple type by work of Taubes.  Moreover,  in \cite{feehanlenessearlier} Feehan and Leness, building on the examples in \cite{fintushelparkstern} with $c_1^2 = \chi_h -3$,  observed that there are \textit{non-minimal} symplectic $4$--manifolds at all the lattice points with $0  \leq c_1^2 <\chi_h -3$ for which the bound on the number of basic classes is sharp.  While it is beyond the scope of our work here,  it would be interesting to determine if the same can be true for minimal examples,  and in particular seeing if the number of basic classes of $Z_{a,b}$ with $0 \leq b  < a -3$ is $a-b-2$.
\end{remark}

\begin{remark}\label{LFgeography}
Although there is  an extensive literature on the geography of (semi-stable) holomorphic fibrations on compact complex surfaces,  the situation in the symplectic case has not been understood well, at least until very recently; see e.g.  \cite{stipsiczgeography, mondengeography}. The recent works of \cite{monden,  cengelkorkmaz} showed the existence of simply-connected Lefschetz fibrations violating  Xiao's famous slope inequality, which is always satisfied by holomorphic fibrations.  Most recently,  it was established in \cite{baykurhamada} that there are simply-connected Lefschetz fibrations with positive signatures---whereas it is still not known if such examples exist in the holomorphic case.  Our examples in this article demonstrate further contrast between the geography of complex and symplectic Lefschetz fibrations. 
\end{remark}

\bigskip
% ============================================================
\section{A Minimal Symplectic Exotic $\CP\#5\CPb$}\label{sec:exotic}

Consider the Lefschetz fibration with monodromy factorization 
\[(t_1^{14}t_at_b t_4 t_6) \cdot   (t_1 t_3 t_5t_7)^{ t_2^{-1} t_5 t_6^{-1}}\cdot D_6=1 \text{ in } \M(\Sigma_3)\]
given by Lemma \ref{lem:genus=3}. This fibration includes two singular fibers given by the clusters \,$t_1^{14}t_at_b t_4 t_6$\, and \,$(t_1 t_3 t_5t_7)^{ t_2^{-1} t_5 t_6^{-1}}$,  along with the two $(-1)$-sections corresponding to the boundary components $\delta_1$ and $\delta_2$ in Figure \ref{fig:genus=3}. The first singular fiber $F_1$ is comprised of a string of 15 transversely intersecting embedded $(-2)$-spheres and an immersed $(-2)$-sphere that transversely intersects the first and last $(-2)$-spheres of the string. The second singular fiber $F_2$ is comprised of two $(-4)$-spheres intersecting each other transversely four times. The $(-1)$-sections intersect $F_1$ in the first and fifteenth $(-2)$-spheres and intersect $F_2$ in each of the $(-4)$-spheres. This configuration is depicted schematically in Figure \ref{fig:fibration}.

Starting with this configuration of spheres, we will perform 16 blowups and a single blowdown to locate a plumbing $P$ symplectically embedded $\CP\# 32\CPb$ that can be symplectically rationally blown down to a symplectic exotic $\CP\# 5\CPb$. To prove that the blown down manifold $X$ is homeomorphic to $\CP\# 5\CPb$, we will apply Freedman's Theorem. To prove $X$ is not diffeomorphic to $\CP\# 5\CPb$, we will consider its symplectic Kodaira dimension. To this end, we first need to understand the homology classes of the fibers and sections of our Lefschetz fibration on $\CP\# 17\CPb$.

\begin{figure}[h!] 
	\centering
	\begin{subfigure}[b]{.53\textwidth}
		\centering
		\includegraphics[scale=.55]{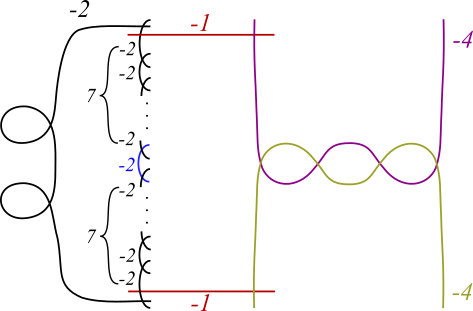}
		\caption{Configuration of spheres in $\CP\# 17\CPb$}\label{fig:fibration}
	\end{subfigure}
	\begin{subfigure}[b]{.46\textwidth}
		\centering
		\includegraphics[scale=.55]{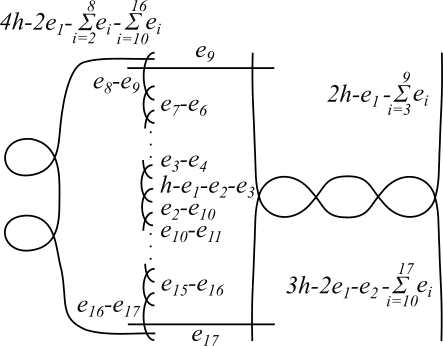}
		\caption{Homology classes}\label{fig:LFhomclasses2}
	\end{subfigure}
	\caption{}
\end{figure}

\begin{prop} Let $\{h,e_1,\ldots,e_{17}\}$ denote the standard basis of $H_2(\CP\# 17\CPb)$. The homology classes of the configuration of spheres in Figure \ref{fig:fibration} are the homology classes shown in Figure \ref{fig:LFhomclasses2}.\label{fig:LFhomclassesprop}\end{prop}

\begin{proof}
	
	To determine these homology classes, we first symplectically blow down the configuration of spheres in $\CP\# 17\CPb$ 16 times. Upon doing so, the ambient 4-manifold is either $\CP\# \CPb$ or $S^2\times S^2$. We first show that it is the latter. Starting with the configuration shown in the first diagram of Figure \ref{fig:blowdowns}, blow down the two $-1$-sections to obtain the next diagram in the figure. Starting with either green $-1$-sphere, sequentially blow down seven times. Then repeat starting with the other green $-1$-sphere. The result will be the configuration of spheres $A$, $B$, $C$, and $D$ depicted in the last diagram. With abuse of notation, let $A$, $B$, $C$, and $D$ denote both the spheres and their homology classes. It is clear that $A^2=12$, $B^2=0$, $C^2=D^2=4$, $B\cdot C=B\cdot D=1,$ and $C\cdot D=4$. Note that when we blew down the green $-1$-spheres in the second diagram, we introduced triple points. We then sequentially blow down at these triple points six times each. Thus we have that $A\cdot C=A\cdot D=7$.
	
	\begin{figure}
		\centering
		\includegraphics[scale=.56]{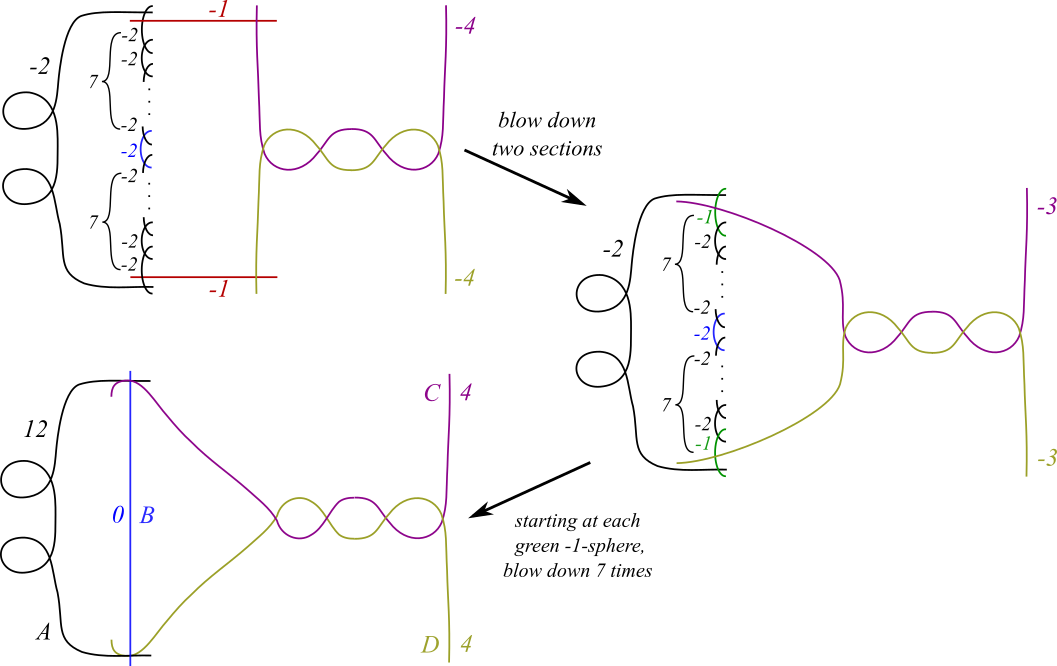}
		\caption{Blowing down to $S^2\times S^2$}\label{fig:blowdowns}
	\end{figure}

	Suppose this configuration lives in $\CP\# \CPb$ and let $\{h,e\}$ be the standard basis for $H_2(\CP\# \CPb)$. Then its canonical class is $K=\text{PD}(-3h+e)$. Let $B=xh+ye$. Then $B^2=x^2-y^2=0$ and $\langle -K, B\rangle=3x+y$. Moreover, by the adjunction equality, $\langle -K, B\rangle = B^2+2-2g(B)=2$. Combining these equations, we find that $x=1, y=-1$ and so $B=h-e$. A similar computation shows that $C=2h$. This implies that $C\cdot B=2$, which is not the case. Thus the configuration of spheres must be in $S^2\times S^2$.
	
	Let $s$ and $f$ denote the standard section and fiber generators of $H_2(S^2\times S^2)$. The canonical class is $K=\text{PD}(-2s-2f)$. Let $A=xs+yf$. Then $A^2=2xy=12$ and $\langle -K,A\rangle = 2x+2y$. Moreover, since the homology class of $A$ can be represented by a genus 2 surface, the adjunction equality gives us $\langle -K,A\rangle= 10$. Combining these equations, we easily see that $(x,y)=(2,3)\text{ or }(3,2)$. We may assume the former so that $A=2s+3f$. Now let $B=xs+yf$. A similar argument shows that $(x,y)=(0,1)\text{ or }(1,0)$. Since $A\cdot B=2$, we must have the former and so $B=f$. Continuing in this way, we see that $C=D=s+2f$.
	
	Now that we know the homology classes of $A$, $B$, $C$, and $D$, we will reverse the blowdown process via blowups to recover the original configuration of spheres in $\CP\# 17\CPb$ along with their homology classes. This blowup process is shown in Figure \ref{fig:LFhomclasses}. The first diagram shows the configuration of spheres $A$, $B$, $C$, $D$ in $S^2\times S^2$. Consider the intersection between $A$, $B$, and $C$, which is marked by a red point. We blow up this point seven times and call the exception spheres $d_1,\ldots,d_7$. The resulting configuration is shown in the second diagram of Figure \ref{fig:LFhomclasses}. Next, blow up seven times in the same way starting at the red point in the second diagram to obtain the third diagram and call the new exceptional spheres $d_9,\ldots,d_{15}$. Finally, blow up at the final two red points to obtain the fourth diagram, which is the original configuration of spheres provided by the Lefschetz fibration. Call these last two exceptional spheres $d_8$ and $d_{16}$. The homology classes of the spheres in this final configuration is shown in the fourth diagram of Figure \ref{fig:LFhomclasses}. 
	
	Let $\{h,e_1,\ldots,e_{17}\}$ be the standard basis for $H_2(\CP\# 17\CPb)$. By performing the change of basis $f\mapsto h-e_1$, $s\mapsto h-e_2$, $d_1\mapsto h-e_1-e_2$, and $d_i\to e_{i+1}$ for all $2\le i\le 16$, we obtain the homology classes shown in Figure \ref{fig:LFhomclasses2}.
\end{proof}

\begin{figure}
	\centering
	\includegraphics[scale=.55]{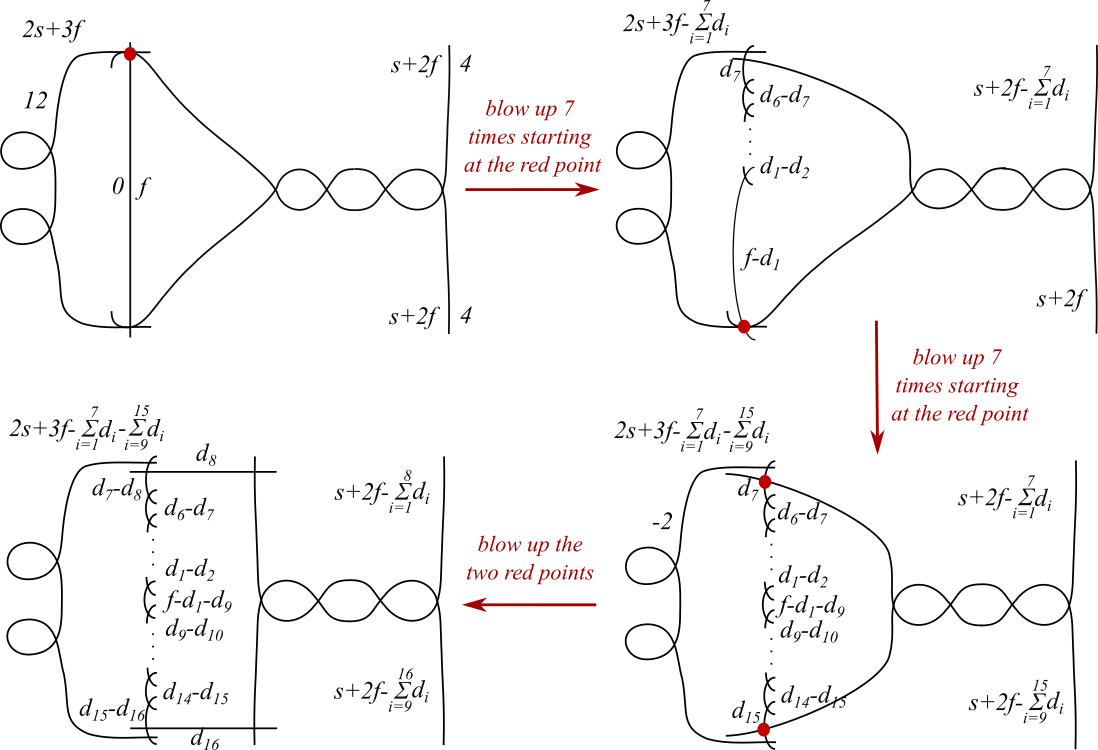}
	\caption{Blowing up $S^2\times S^2$ to recover the homology classes of the original Lefschetz fibration configuration}\label{fig:LFhomclasses}
\end{figure}

\begin{figure}
	\centering
	\includegraphics[scale=.53]{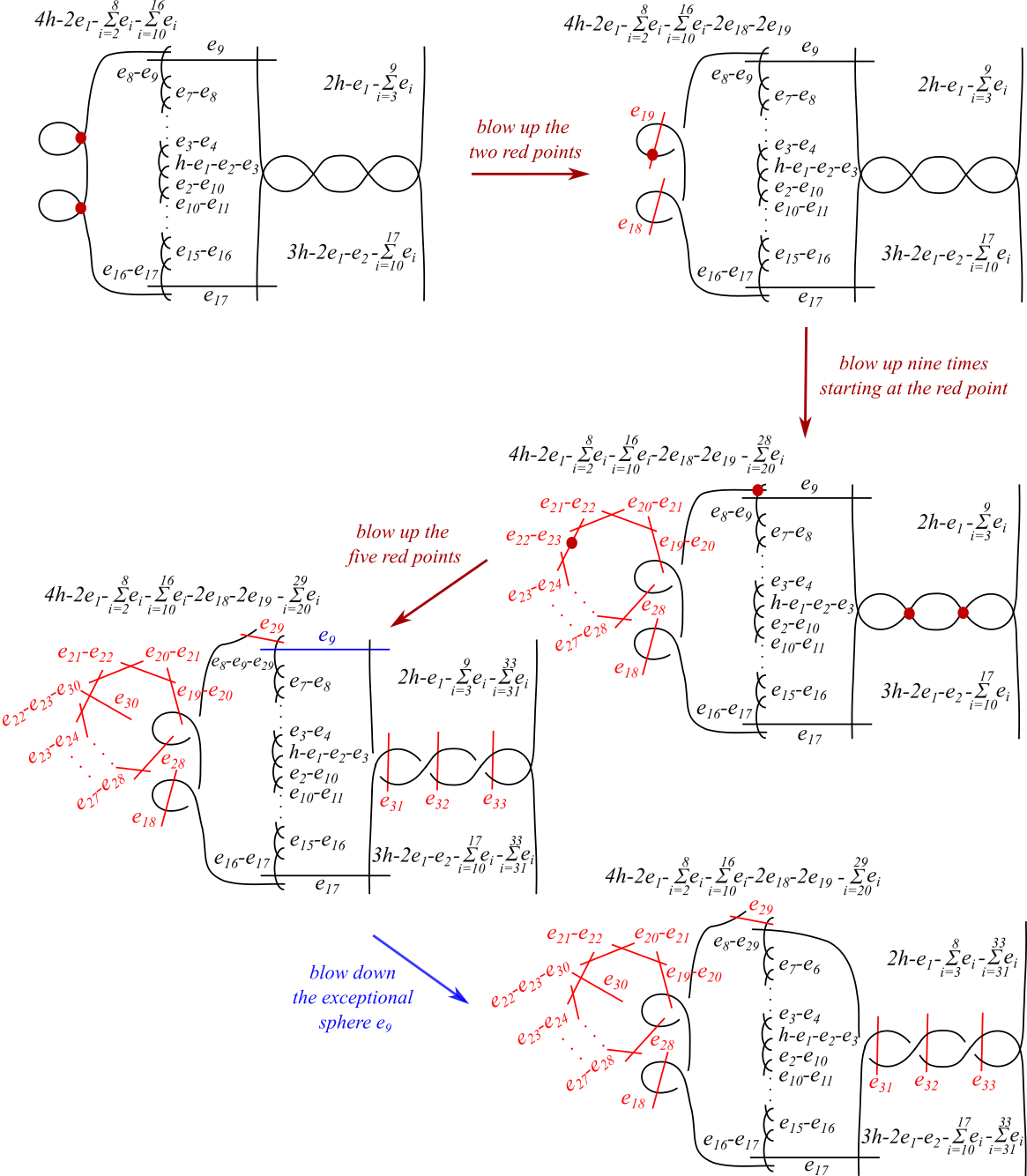}
	\caption{Sixteen blowups and one blowdown.}\label{fig:blowupfibration}
\end{figure}

\begin{figure}
	\centering
	\begin{subfigure}{\textwidth}
		\centering
		\includegraphics[scale=.5]{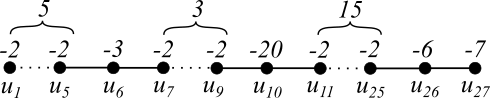}
		\bigskip
		\caption{The linear plumbing $P$}\label{fig:plumbing}
		\vspace{1cm}
	\end{subfigure}
	
	\begin{subfigure}{\textwidth}
		\centering
		\begin{tabular}{ c | c || c| c }
			Sphere & Homology Class & Sphere & Homology Class\\\hline
			$u_1$ & $e_{27}-e_{28}$ & $u_{15}$ & $e_{12}-e_{13}$\\\hline
			$u_2$ & $e_{26}-e_{27}$ & $u_{16}$ & $e_{11}-e_{12}$\\\hline
			$u_3$ & $e_{25}-e_{26}$ & $u_{17}$ & $e_{10}-e_{11}$\\\hline
			$u_4$ & $e_{24}-e_{25}$ & $u_{18}$ & $e_{2}-e_{10}$\\\hline
			$u_5$ & $e_{23}-e_{24}$ & $u_{19}$ & $h-e_1-e_2-e_3$\\\hline
			$u_6$ & $e_{22}-e_{23}-e_{30}$ & $u_{20}$ & $e_{3}-e_{4}$\\\hline
			$u_7$ & $e_{21}-e_{22}$ & $u_{21}$ & $e_{4}-e_{5}$\\\hline
			$u_8$ & $e_{20}-e_{21}$ & $u_{22}$ & $e_{5}-e_{6}$\\\hline
			$u_9$ & $e_{19}-e_{20}$ & $u_{23}$ & $e_{6}-e_{7}$\\\hline
			$u_{10}$ & $\displaystyle 4h-2e_1-\sum_{i=2}^8 e_i - \sum_{i=10}^{16}e_i-2e_{18}-2e_{19}-\sum_{i=20}^{29}e_i$ &$u_{24}$ & $e_{7}-e_{8}$ \\\hline
			$u_{11}$ & $e_{16}-e_{17}$ & $u_{25}$ & $e_{8}-e_{29}$\\\hline
			$u_{12}$ & $e_{15}-e_{16}$&  $u_{26}$ & $\displaystyle 2h-e_1-\sum_{i=3}^8 e_i - \sum_{i=31}^{33}e_i$\\\hline
			$u_{13}$ & $e_{14}-e_{15}$& $u_{27}$ & $\displaystyle 3h-2e_1-e_2-\sum_{i=10}^{17} e_i - \sum_{i=31}^{33}e_i$ \\\hline
			$u_{14}$ & $e_{13}-e_{14}$
		\end{tabular}
		\bigskip
		\caption{The homology classes of $u_1,\ldots,u_{27}$}\label{fig:homclasses}
	\end{subfigure}
	\medskip
	\caption{The plumbing $P$ and its homology classes}\label{fig:Pclasses}
\end{figure}

Let $P$ be the linear plumbing with weights $$(\underbrace{-2,\ldots,-2}_{5},-3,\underbrace{-2,\ldots,-2}_{3}, -20, \underbrace{-2,\ldots,-2}_{15},-6,-7).$$ 

By symplectically blowing up the Lefschetz fibration on $\CP\# 17\CPb$ sixteen times, and blowing down once, we can find $P$ symplectically embedded in $\CP\# 32\CPb$. This process is shown in Figure \ref{fig:blowupfibration}. Let $u_1,\ldots, u_{27}$ denote the homology classes of the spheres of $P$, as shown in Figure \ref{fig:plumbing}. These homology classes are found explicitly in Figure \ref{fig:blowupfibration}, but for simplicity, they are recorded in the table in Figure \ref{fig:homclasses}.

The boundary of $P$ is the lens space $L(585^2,291914)$, which bounds a rational homology 4-ball $B$, by Lisca \cite{liscalensspace}. Moreover, by Park \cite{parkbdown} and Symington \cite{symington2}, $P$ can by symplectically rationally blown down. Let $Z=\CP\# 32\CPb\setminus P$ and let $X=Z\cup B$ be the result of the symplectic rational blowdown.

\medskip
\begin{prop} $X$ is homeomorphic to $\CP\# 5\CPb$.\label{homeomorphic}\end{prop}

\begin{proof} We first show that $X$ is simply connected. By the Seifert Van-Kampen Theorem, $\pi_1(X)=\pi_1(Z)\ast_{\pi_1(\partial P)}\pi_1(B)$. Since $\CP\# 32\CPb$ is simply connected, the map $\pi_1(\partial P)\to \pi_1(Z)$ induced by inclusion is surjective. Moreover, the map $\pi_1(\partial P)\cong\Z_{585^2}\to \pi_1(B)\cong \Z_{585}$ induced by inclusion is surjective (\cite{parkbdown}). Thus, it suffices to show that $Z$ is simply connected. Now since $\pi_1(\partial P)$ is abelian, so is $\pi_1(Z)$; consequently, $\pi_1(Z)\cong H_1(Z)$ and so it suffices to show that $H_1(Z)$ is trivial.
	
	For $1\le i\le 27$, let $\mu_i\in H_1(\partial P)$ denote the homology class of the meridian of the $i$th surgery curve in the obvious surgery diagram for $\partial P$. Let $\overline{\mu_i}$ denote the image of $\mu_i$ in $H_1(Z)$. Then $H_1(\partial P)\cong \Z_{585^2}$ is generated by $\mu_1$ and the elements $\mu_1,\ldots, \mu_{10}$ satisfy the relations
	$$ 2\mu_1 = \mu_2,\quad 2\mu_i=\mu_{i-1}+\mu_{i+1}\text{ for $2\le i\le 5$ and $6\le i\le 9$}, \text{ and} \quad 3\mu_6=\mu_5+\mu_7.$$
	Combining these relations, we have that $\mu_6=6\mu_1$ and $\mu_{10}=34\mu_1$. Since $H_1(\partial P)$ surjects onto $H_1(Z)$, $H_1(Z)$ is generated by $\overline{\mu_1}$ and satisfies the relations $\overline{\mu_6}=6\overline{\mu_1}$ $\overline{\mu_{10}}=34\overline{\mu_1}$.
	
	In $Z$, $\overline{\mu_6}$ can be represented by the equator of the exceptional sphere $e_{30}$; thus $\overline{\mu_6}=0$ and so $6\overline{\mu_1}=\overline{\mu_6}=0$. Similarly, $\overline{\mu_1}$ and $\overline{\mu_{10}}$ can be represented by circles on the exceptional sphere $e_{29}$ which cobound an annulus; hence $\overline{\mu_1}=-\overline{\mu_{10}}$ and so $\overline{\mu_1}=\overline{\mu_{10}}=34\overline{\mu_1}$, or $35\overline{\mu_1}=0$. Since 6 and 35 are relatively prime, we readily obtain $\overline{\mu_1}=0$, proving that $H_1(Z)$ is trivial. Thus $X$ is simply connected.
	
	Now, $\sigma(X)=\sigma(\CP\# 32\CPb)-\sigma(P)+\sigma(B)=-4=\sigma(\CP\# 5\CPb)$ and $\chi(X)=\chi(\CP\# 32\CPb)-\chi(P)+\chi(B)=\chi(\CP\# 5\CPb)$. Since the signatures of $X$ and $\CP\# 5\CPb$ are not divisible by 16, both manifolds are odd. Thus, by Freedman's theorem \cite{freedman}, $X$ is homeomorphic to $\CP\# 5\CPb$.
\end{proof}

\begin{prop} $X$ is not diffeomorphic to $\CP\# 5\CPb$.\label{diffeomorphic}\end{prop}
\begin{proof}
	Let $K_{\CP\# 5\CPb}$ denote the canonical class of $\CP\# 5\CPb$. By Theorem D in \cite{liliu95}, $\CP\# 5\CPb$ has a unique symplectic structure up to diffeomorphism and deformation. Consequently, since the standard symplectic structure $\nu$ on $\CP\# 5\CPb$ satisfies $(K_{\CP\# 5\CPb})\cdot\nu<0$, $\CP\# 5\CPb$ does not admit a symplectic structure $\nu$ satisfying $(K_{\CP\# 5\CPb})\cdot\nu>0$. Let $K_X$ denote the canonical class of $X$ and let $\omega_X$ denote the symplectic class of $X$. We will show that $K_X\cdot \omega_X>0$, proving that $X$ is not diffeomorphic to $\CP\# 5\CPb$.
	
	The canonical class of $\CP\# 32\CPb$ is given by $K=\text{PD}(-3h+\displaystyle\sum^{33}_{\begin{subarray}{l} i=1\\i\neq 9\end{subarray}}e_i)$. It is well-known that $\CP\# 32\CPb$ admits a symplectic structure compatible with $K$ and whose cohomology class can be represented by $\displaystyle\omega=PD(ah-\sum_{\begin{subarray}{l} i=1\\i\neq 9\end{subarray}}^{33}b_ie_i)$, where $a>b_0>\cdots>b_{33}$ and $a>\displaystyle\sum^{33}_{\begin{subarray}{l} i=1\\i\neq 9\end{subarray}}b_i$ (see, for example, Lemma 5.4 in \cite{karakurtstarkston}).  Note that $K\cdot\omega= -3a+\displaystyle\sum^{32}_{\begin{subarray}{l} i=1\\i\neq 9\end{subarray}}b_i$.
	
	Give $H_2(P)$ the basis $\{u_1,\ldots,u_{27}\}$ and give $H^2(P)$ the hom-dual basis $\{\gamma_1,\ldots,\gamma_{27}\}$. Then
	\noindent $K|_P=\displaystyle\sum_{i=1}^{27}(K\cdot u_i)\gamma_i=\gamma_6+18\gamma_{10}+4\gamma_{26}+5\gamma_{27}$ and 
	
	\begin{flalign*}
	\omega|_P=\displaystyle\sum_{i=1}^{27}(\omega\cdot u_i)\gamma_i = &(b_{27}-b_{28})\gamma_1 + (b_{26}-b_{27})\gamma_2  + \cdots + (b_{23}-b_{24})\gamma_5 + &\\
	&(b_{22}-b_{23}-b_{30})\gamma_6 + (b_{21}-b_{22})\gamma_7 + \cdots + (b_{19}-b_{20})\gamma_9 +\\
	& (\displaystyle 4a-2b_1-\sum_{i=2}^8 b_i - \sum_{i=10}^{16}b_i-2b_{18}-2b_{19}-\sum_{i=20}^{29}b_i)\gamma_{10} + \\
	&(b_{16}-b_{17})\gamma_{11} + \cdots + (b_{10}-b_{11})\gamma_{17}  + (b_{2}-b_{10})\gamma_{18} +\\
	& (a-b_{1}-b_{2}-b_3)\gamma_{19} + (b_{3}-b_{4})\gamma_{20} + \cdots +  (b_{7}-b_{8})\gamma_{24} +  (b_{8}-b_{29})\gamma_{25} + \\
	&( 2a-b_1-\sum_{i=3}^8 b_i - \sum_{i=31}^{33}b_i)\gamma_{26} +  (3a-2b_1-b_2-\sum_{i=10}^{17} b_i - \sum_{i=31}^{33}b_i)\gamma_{27}.
	\end{flalign*}	
	
	Let $Q$ denote the matrix for the intersection form of $P$ with respect to the basis $\{u_1,\ldots,u_{27}\}$. Then 	
	\begin{flalign*}
	K|_P\cdot \omega|_P= Q^{-1}(K|_P,\omega|_P)= \displaystyle\frac{1}{585}\Big(&-5544a + 3309b_1+1082(b_2+b_{10}+\cdots+b_{17}) &\\
	&+1153(b_3+\cdots+b_8+b_{29})+1168b_{18} +601(b_{19}+\cdots+b_{22})\\
	&+670(b_{23}+\cdots+b_{28})+516b_{30}+1067(b_{31}+\cdots +b_{33})\Big)\\
	\end{flalign*}
	Finally, we have that
	\begin{flalign*}
	K_X\cdot \omega_X&= (K_X)|_Z\cdot (\omega_X)|_Z = K|_Z\cdot\omega|_Z= K\cdot\omega - K|_P\cdot \omega|_P  &\\
	&= \frac{1}{585}\Big(3789a-2724b_1-497(b_2+b_{10}+\cdots+b_{17})-568(b_3+\cdots+b_8+b_{29})-583b_{18}\\
	&\qquad\qquad-16(b_{19}+\cdots+b_{22})-85(b_{23}+\cdots+b_{28})+69b_{30}-482(b_{31}+\cdots +b_{33})\Big)\\
	&>\frac{1}{585}(a-\displaystyle\sum^{33}_{\begin{subarray}{l} i=1\\i\neq 9\end{subarray}}b_i)>0.\\
	\end{flalign*}	
	\end{proof}

\enlargethispage{0.2in}
\begin{prop} $X$ is minimal.\label{prop:minimal}\end{prop}
\begin{proof}
	Following the strategy of Ozsv\'{a}th-Szab\'{o} in \cite{ozsvathszabo}, we will show that $X$ has a unique basic class (up to sign), which implies that $X$ is minimal. The following calculations were completed in Matlab. Because the efficacy of this type of computation is well-documented (c.f.  \cite{ozsvathszabo}, \cite{stipsiczszabo},\cite{karakurtstarkston}), we will only highlight the main steps.
	
We first find the following basis for $H_2(Z)$: 
$$A_1=h-e_1-e_{18}-e_{31},\quad A_2=h-e_1-e_{18}-e_{32},\quad A_3=h-e_1-e_{18}-e_{33},$$ 
$$A_4=3h-3e_1+2e_{18}-\sum_{i=19}^{22}2e_i-2e_{30}-e_{31}-e_{32}-e_{33}$$
$$A_5=3e_{18}-\sum_{i=23}^{28}e_i+e_{30}$$
$$A_6=36h-21e_1-7e_2-\sum_{i=3}^{8}8e_i-\sum_{i=10}^{17}7e_i+2e_{18}+\sum_{i=23}^{28}e_i-8e_{29}-e_{30}-e_{31}-e_{32}-e_{33}$$

Notice that $A_1,A_2,A_3$ and $A_5$ can be represented by spheres, $A_4$ can be represented by a torus, and $A_6$ can be represented by a surface of genus 387 (\cite{linonnegclasses}). Moreover, $A_1^2, A_2^2, A_3^2=-2$, $A_4^2=-27$, $A_5=-16$, and $A_6=-48$. 

We would like to find the number of basic classes $L$ on $X$ satisfying the additional criterion: $|L(A_i)|\le -A_i^2$ and $L(A_i)\equiv A_i^2\pmod 2$. Such classes are called \textit{adjunctive} classes. There are 9,317,700 possible adjunctive classes. If $L$ a basic class, then by the Seiberg-Witten dimension formula, $d=\frac{L^2-3\sigma(X)-2\chi(X)}{4}=\frac{L^2-4}{4}\ge 0$ and $d\equiv 0\pmod 2$. Thus $L^2\ge 4$ and $L\equiv 4\pmod 8$. This restriction leaves us with 13,960 possible adjunctive basic classes.

Now, if $L$ is a basic class on $X$, then by the gluing formula (\cite{parkbdown}), there is a basic class $\tilde{L}$ on $\CP\#32\CPb$ inducing $L$ such that $(\tilde{L}|_P)^2=-27$. Moreover, the set of basic classes on $\CP\#32\CPb$ inducing $L$ contains an element $\tilde{L}$ such that $u_i^2+2\le \tilde{L}_P(u_i) \le -u_i^2$ for all $1\le i\le 27$ (where $u_i$ are the homology classes of the spheres in $P$ shown in Figure \ref{fig:homclasses}). There are 585 classes on $P$ satisfying these conditions. Thus we have a total of 8,166,600 classes on $\CP\#32\CPb$ that could give rise to adjunctive basic classes on $X$. These are given by $(L|_Z,\tilde{L}|_P)$. 

Let $H=(2,0,0,0,1,1,0,\ldots,0)\in H_2(\CP\#32\CPb;\Q)$, written in the basis 
$\{A_1,\ldots,A_6,u_1,\ldots,u_{27}\}$. Then $H(u_i)=0$ for all $i$, $H^2>0$, and $H\cdot PD(h)>0$. 
By the wall-crossing formula, if $\tilde{L}$ is a basic class on $\CP\#32\CPb$ inducing a basic class on $X$, then $\text{sign}(\tilde{L}\cdot H)\neq\text{sign}(\tilde{L}\cdot h)$. There are 1788 such classes. Finally, we find that only 2 of these 1788 classes are integral. They are $\pm K$ (the canonical class). Consequently, $X$ has a two adjunctive basic classes.

Finally following the argument in \cite{ozsvathszabo}, it is easy to see that every basic class on $X$ must be adjunctive. Thus $X$ has a unique basic class (up to sign).
\end{proof}

\medskip
\begin{remark} \label{improve}
It is worth highlighting that we derived our rational blowdown configuration only using the first two clusters in the positive factorization $(t_1^{14} \, t_at_b t_4 t_6) \cdot   (t_1 t_3 t_5t_7)^{ t_2^{-1} t_5 t_6^{-1}}
	\cdot D_6 = t_{\delta_1}t_{\delta_2}$ in $\M(\Sigma_3^2)$, where the key point is the topological configuration of the disjoint Dehn twist curves in each cluster. Curiously, if one can get a similar positive factorization but with another cluster,  such as a factor \textit{conjugate} to $(t_a t_3 t_5t_7)$,  rationally blowing down a disjoint $(-4)$-sphere there would yield an exotic $\CP \# \, 4\CPb$.  In our factorization,  we can find two of the desired Dehn twist factors in $D_6$, but not all four.
\end{remark}

\clearpage
\vspace{0.1in}

\bibliographystyle{plain}
\bibliography{Bibliography.bib}

\end{document}